\documentclass{amsart}
\usepackage{amsfonts,  bbm}
\usepackage{graphicx}
\usepackage{overpic}
\usepackage{epsfig} 
\usepackage{hyperref}
\usepackage{epstopdf,MnSymbol}  

\newtheorem{theorem}{Theorem}[section]
\theoremstyle{plain}

\newtheorem{lemma}[theorem]{Lemma}

\theoremstyle{remark}

\numberwithin{equation}{section}

\newcommand{\dist}{\operatorname{dist}}
\newcommand{\diam}{\operatorname{diam}}

\newcommand{\calH}{{\mathcal H}}

\newcommand{\la}{\lambda}

\newtheorem*{prop*}{Theorem}

\newcommand{\vep}{\varepsilon}

\newcommand{\norm}[1]{\left\Vert #1 \right\Vert}
\newcommand{\abs}[1]{\left|#1 \right|}

\newcommand{\nvert}{n_V}
\newcommand{\nedge}{n_E}

\def\ins#1#2#3{\vbox to0pt{\kern-#2 \hbox{\kern#1 #3}\vss}\nointerlineskip}

\begin{document}

\title[Sharp diameter bound]{Sharp diameter bound on the spectral gap for quantum graphs}
\author[Borthwick]{David Borthwick}
\address{Department of Mathematics, Emory University, Atlanta, GA 30322}
\email{dborthw@emory.edu}
\author[Corsi]{Livia Corsi}
\address{Department of Mathematics, Emory University, Atlanta, GA 30322}
\email{livia.corsi@emory.edu}
\author[Jones]{Kenny Jones}
\address{Department of Mathematics, Emory University, Atlanta, GA 30322}
\email{wesley.kenderdine.jones@emory.edu}
\date{\today}

\begin{abstract}
We establish an upper bound on the spectral gap for compact quantum graphs which depends only on the diameter and total number of vertices. This bound is asymptotically sharp for pumpkin chains
with number of edges tending to infinity.
\end{abstract}

\maketitle

\section{Introduction}

A quantum graph $G$ is a metric graph equipped with a quantum Hamiltonian operator $L$ acting on the $L^2$ space of functions on the edges of $G$. For the main result of this paper, $L = -d^2/dx^2$, with standard Kirchoff-Neumann boundary conditions.
This means that $L$ is the self-adjoint operator associated to the quadratic form,
\[
\calH(f) :=  \norm{f'}^2,
\]
on $H^1(G)$.  The domain of $L$ consists of continuous functions on $G$ which are $H^2$ on each individual edge,
such that the sum of outgoing derivatives at each vertex vanishes.  See \cite{BK:2013} for background on quantum
graphs and self-adjoint vertex conditions.

We assume that $G$ is compact, meaning that it has a finite number of edges, all of finite length.  
This implies that the lowest eigenvalue $\lambda_0 = 0$ is simple.  The \emph{spectral gap} of $G$ is the lowest nonzero eigenvalue,
which we denote by $\lambda_1$.  By the min-max principle for the quadratic form $\calH$, 
we have Rayleigh quotient formula for the first nonzero eigenvalue
\begin{equation}\label{lam1.rayleigh}
\lambda_1 = \inf \left\{ \frac{\norm{f'}^2}{\norm{f}^2}:\>f \in H^1(G), \int_G f = 0 \right\}.
\end{equation}
The higher eigenvalues $\lambda_n$ are similarly obtained from variational formulas.

A variety of results have been proven which estimate the eigenvalues in terms of basic geometric properties of $G$.  For example,
Friedlander \cite{Friedlander:2005} proved the lower bound
\[
\lambda_n \ge \left( \frac{\pi (n-1)}{2\ell(G)} \right)^2,
\]
for a connected compact graph, where $\ell(G)$ is the total edge length of $G$.  The $n=1$ estimate is optimal
in the case of a single interval, and for $n \ge 2$ the estimate is optimal for an equilateral star graph with $n-1$ edges.  
The method of proof involves symmetrization after reducing the argument to trees, 
Independent proofs of the lower bound on $\lambda_1$ were given in \cite{Nicaise:1987, KN:2014}.

No upper bound for $\lambda_1$ is possible in terms of $\ell(G)$ alone. For instance if we consider a flower graph with fixed total length, 
as the number of edges $\nedge \to \infty$,
we see $\lambda_1 \to \infty$.   If $\nedge$ is fixed, then Kennedy et al.~\cite{KKMM:2016} proved that
\[
\lambda_1 \le \left( \frac{\pi \nedge}{\ell(G)} \right)^2.
\]
This bound is sharp in the case of flower graphs or equilateral pumpkins.  (A \emph{pumpkin}, also called a dipole graph,
consists a pair of vertices connected by a set of parallel edges.)  Additional bounds on the eigenvalues $\lambda_n$ can be
found in \cite{BG:2017, BK:2012, BKKM:2017, BKKM:2018, KKMM:2016, Kennedy:2018, KMN:2013}.

In this paper we focus on the issue of controlling $\lambda_1$ from above in terms of the number of vertices, $\nvert$, and 
the diameter,
\[
\diam(G) := \sup \{ \text{dist}(x,y):\>x,y\in G\},
\]
where the distance is the length of the shortest path connecting two points within the graph. 
Kennedy et al.~also showed it is not possible to bound $\lambda_1$ from above in terms of $\diam(G)$ alone. 
However, they provided the following upper bound for $\lambda_1$ in terms of $\diam(G)$ and the total number of vertices, 
$\nvert$ (\cite{KKMM:2016}--Thm.~6.1):  For $\nvert \geq 2$
\begin{equation}\label{lam1.kkmm}
\lambda_1 \le  \left(\frac{\pi(\nvert+1)}{\diam(G)}\right)^2.
\end{equation}
In terms of the combinatorial diameter $\diam_V(G)$, defined as the maximal distance between two vertices of the graph,
the estimate is improved to 
\begin{equation}\label{lam1.kkmm2}
\lambda_1 \le \left( \frac{\pi(\nvert-1)}{\diam_V(G)} \right)^2.
\end{equation}
The estimate \eqref{lam1.kkmm2} is clearly sharp in the case of a single interval, or equilateral pumpkin.
In a remark following the theorem, the authors of  \cite{KKMM:2016} observe that these estimates are not optimal in general,
and conjecture that the optimal bound will be sharp only in an asymptotic sense.

We will prove the following sharp version of the estimates \eqref{lam1.kkmm} and \eqref{lam1.kkmm2}:
\begin{theorem}\label{main.thm}
For a compact quantum graph $G$ with $\nvert \ge 2$ vertices and diameter $\diam(G)$, the first nonzero eigenvalue satisfies
\begin{equation}\label{main.bnd1}
\lambda_1 \le \left( \frac{\pi(\nvert+2)}{2\diam(G)} \right)^2.
\end{equation}
In terms of the combinatorial diameter,
\begin{equation}\label{main.bnd2}
\lambda_1 \le \left( \frac{\pi \nvert}{2\diam_V(G)} \right)^2.
\end{equation}
\end{theorem}

The proof makes use of the pumpkin-chain reduction argument introduced in \cite{KKMM:2016}.  A \emph{pumpkin chain}
is a graph consisting of a linear arrangement of equilateral pumpkins, joined at the vertices.  Each pumpkin in the chain can have
a different number of edges.  The diameter of a pumpkin chain is simply the distance between the endpoints of the chain.

%%%%%%%%%%%%%%%%%%%%%%%%%%%%%%%%%%%%%%%%%%%%%%%%%%%%%%%%%%%%%%%%%%%%%%%% 
% Figure 1 
%%%%%%%%%%%%%%%%%%%%%%%%%%%%%%%%%%%%%%%%%%%%%%%%%%%%%%%%%%%%%%%%%%%%%%%% 
\begin{figure}[ht] 
%\vskip.3truecm 
\centering 
\includegraphics[width=4in]{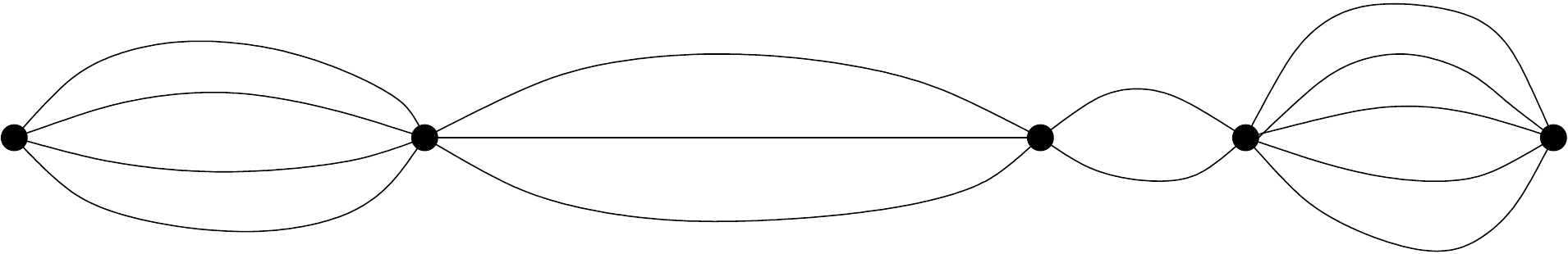} 
\vskip.2truecm 
\caption{A pumpkin chain. The edges of each pumpkin have the same length.} 
\label{fig:0} 
\end{figure} 
%%%%%%%%%%%%%%%%%%%%%%%%%%%%%%%%%%%%%%%%%%%%%%%%%%%%%%%%%%%%%%%%%%%%%%%% 

It was shown in \cite{KKMM:2016} that an arbitrary graph can be reduced to a pumpkin chain without decreasing
$\lambda_1$, while controlling $\diam(G)$ and $\nvert$.  This argument is reviewed in \S\ref{reduction.sec}.
The result \eqref{lam1.kkmm} was deduced from a test function constructed on the longest pumpkin in the chain.  
To prove Theorem~\ref{main.thm}, we prove an optimal bound on the spectral gap form a pumpkin chain in \S\ref{pchain.bnd.sec}.
The strategy is to optimize over a choice of two test functions, 
one constructed on the longest pumpkin and the other on the longest two pumpkins.

The estimate \eqref{main.bnd2} agrees with \eqref{lam1.kkmm2} when $\nvert = 2$, and is therefore sharp in the case of
an equilateral pumpkin, as noted above.
For $\nvert \ge 3$, we will demonstrate in \S\ref{sharp.sec} that \eqref{main.bnd2}  is sharp in an asymptotic sense.  
That is, for all $\nvert\ge 3$ and $\epsilon>0$ there exists a graph for which 
\[
\lambda_1 \ge \left( \frac{\pi(\nvert+1)}{2\diam(G)} \right)^2 - \epsilon.
\]
These examples are constructed from pumpkin chains with fixed $\nvert$ and $\diam(G)$, and $\nedge \to \infty$.
For a given $\nvert$, $\lambda_1$ is maximized when a single pumpkin has twice the length of the remaining pumpkins, which have uniform length. The location of the large pumpkin in the chain has no effect on the optimal bound.  

To understand the configuration of the optimal chain, let 
$a = \diam(G)/\nvert$ and consider a chain with $\nvert$ segments of length $a$ and one of length $2a$.  The bound
\eqref{main.bnd2} gives $\lambda_1 \le (\pi/2a)^2$, so an optimal eigenfunction has frequency close
to $\pi/2a$.  Therefore, in each pumpkin of length $a$, the eigenfunction nearly
passes through a quarter wavelength.  Assuming the eigenfunction is positive at the left endpoint and negative at the right, 
it will have components which are concave down on the left side and concave down on the right.  The single pumpkin of double length
is needed to contain the inflection point where the concavity changes sign.

%%%%%%%%%%%%%%%%%%%%%%%%%%%%%%%%%%%%%%%%%%%%%%%%

%%%%
\section{Reduction to pumpkin chains}\label{reduction.sec}

In this section we review the argument in \cite{KKMM:2016} on the reduction to pumpkin chains. 
As we shall see, in order to reduce to pumpkin-chains one needs to perform
operations on the graph, such as cutting pendants, shortening edges and identify vertices.
Thus it is clearly important to understand how such operations affect $\la_1$. The following is a well known result; see, for instance,
\cite{BK:2013,KKMM:2016,KMN:2013}.

\begin{lemma}\label{known}
Suppose that $G$ and $G'$ are connected, compact and finite quantum grahps. If $G'$ can be obtained from $G$ by either
cutting a pendant, shortening an edge or identifying two vertices, then $\la_1(G')\ge \la_1(G)$.
\end{lemma}

The pumpkin chain algorithm which serves as the basis for the estimate of the spectral gap is described in the following:
\begin{lemma}[\cite{KKMM:2016}--Lemma~5.4]\label{pchain.lemma}
Given a compact, connected, non-empty metric graph $G$, there exists a pumpkin chain $G^*$ such that 

\begin{enumerate}

\item $\diam(G)=\diam(G^*)$, $\ell(G) \geq \ell(G^*)$, and $\nvert(G) \geq \nvert(G^*)-2$.

\item $\lambda_1(G) \leq \lambda_1(G^*)$.

\end{enumerate}
In the combinatorial diameter is used, (1) is replaced by
\begin{itemize}
\item[($1'$)]  $\diam_V(G)=\diam_V(G^*)$, $\ell(G) \geq \ell(G^*)$, and $\nvert(G) \geq \nvert(G^*)$.
\end{itemize}
\end{lemma}

\begin{proof}
The pumpkin chain $G^*$ is produced by the following algorithm, where all steps either increase $\lambda_1$ or leave it unchanged.

\begin{itemize}

\item[Step 1.] Let $D := \diam(G)$, and choose two points $x,\,y \,\in G$ such that $\dist(x,y)=D$. If $x$ and $y$ are not vertices, 
we can insert artificial vertices of degree two at $x$ and $y$.  This yields the shift of $\nvert$ by $2$ in statement (1).
If $x$ and $y$ are assumed to be vertices, then no change in $\nvert$ is required and (1) is modified to ($1'$).
For clarity we will rename $x$ and $y$ as the vertices $v_0$ and $v_D$ respectively. 
 
 \item[Step 2.]
Choose a path $\Gamma_1$ of minimal distance between $v_0$ and $v_D$, so that $\ell(\Gamma_1)=D$. 
Note that, since $\Gamma_1$ is the shortest path connecting $v_0$ and $v_D$, 
it contains no loops and does not cross any point twice (vertex or edge).  In the example shown in Figure~\ref{fig1},
$\Gamma_1$ is the central path. 

%%%%%%%%%%%%%%%%%%%%%%%%%%%%%%%%%%%%%%%%%%%%%%%%%%%%%%%%%%%%%%%%%%%%%%%% 
% Figure 1 
%%%%%%%%%%%%%%%%%%%%%%%%%%%%%%%%%%%%%%%%%%%%%%%%%%%%%%%%%%%%%%%%%%%%%%%% 
\begin{figure}[ht] 
%\vskip.3truecm 
\centering 
\ins{112pt}{-60pt}{$v_0$} 
\ins{277pt}{-60pt}{$v_D$} 
\includegraphics[width=2.1in]{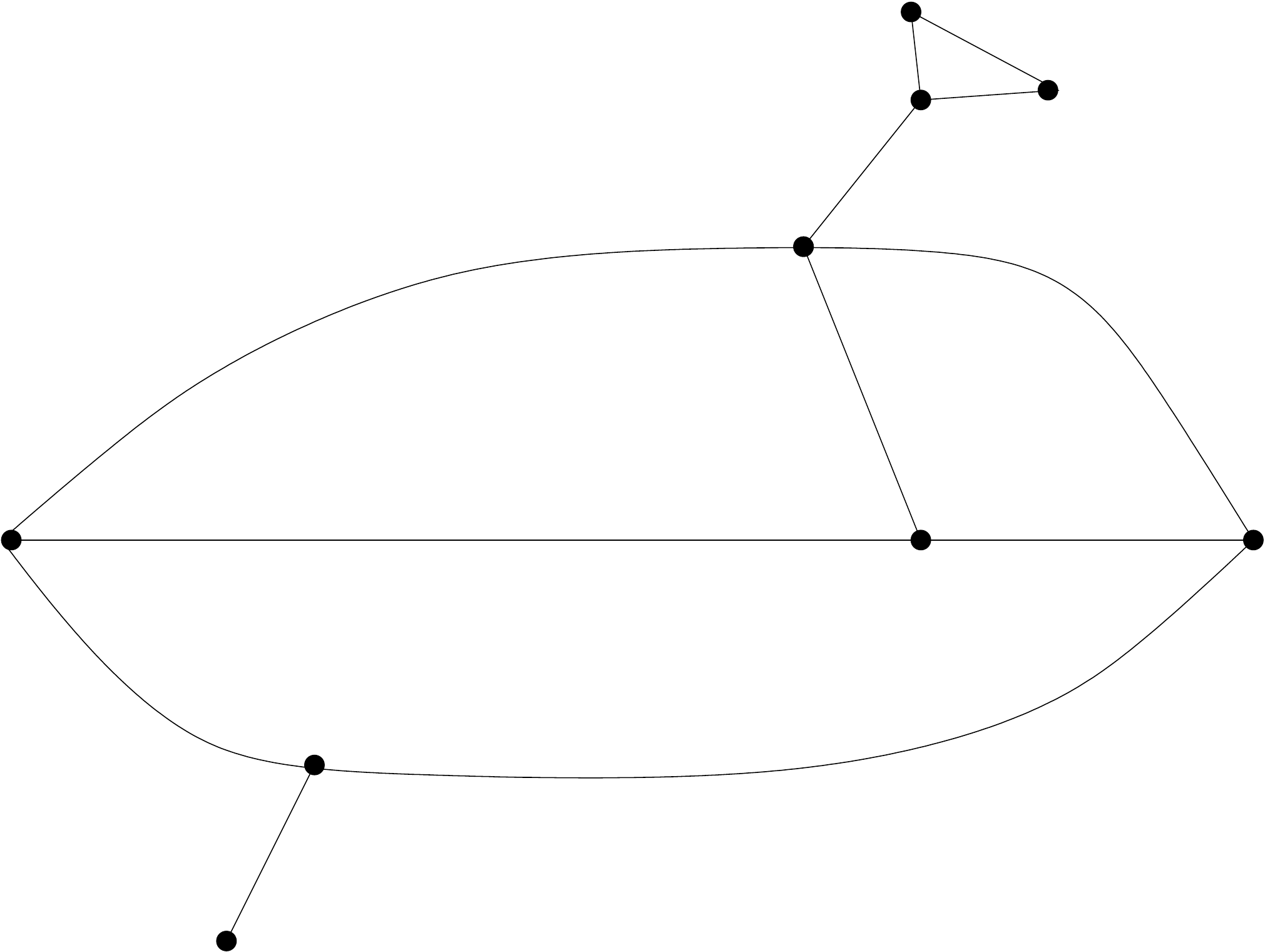} 
\vskip.2truecm 
\caption{The graph $G$.} 
\label{fig1} 
\end{figure} 
%%%%%%%%%%%%%%%%%%%%%%%%%%%%%%%%%%%%%%%%%%%%%%%%%%%%%%%%%%%%%%%%%%%%%%%% 

 \item[Step 3.] Find the second shortest path $\Gamma_2$, such that $\Gamma_2$ does not cross any point $x \in G$ twice and $\Gamma_2 \nsubseteq \Gamma_1$. 
 If two or more paths have the same length and satisfy the above conditions, choose one arbitrarily. If such a path does not exist, skip to Step 5.
 
 \item[Step 4.] Continue to find the next shortest path connecting $v_0$ to $v_D$ such that no point is crossed twice and 
\[\Gamma_k \nsubseteq \bigcup_{i=1}^{k-1}\Gamma_i\]
for each $k$.
By compactness of $G$, this process must terminate at some $\Gamma_n$
with 
\[
D=\ell(\Gamma_1) \leq \ell(\Gamma_2) \leq \dots \leq \ell(\Gamma_{n-1}) \leq \ell(\Gamma_n).
\]

\item[Step 5.]
Let $G' = \bigcup_{i=1}^{n} \Gamma_i$. We claim that any connected component of $G\backslash G'$ must be attached to
$G'$ by a single vertex (i.e., is a pendant of $G'$). Indeed, if that were not the case
 we could find a non-self intersecting path connecting $v_0$ to $v_D$ that intersects $G\backslash G'$ and hence is not a subset of $G'$.
However, we exhausted all such paths in Step 4.   
We can therefore pass from $G$ to $G'$ by removing all pendants attached to $G'$, see Figure \ref{fig2}. 

%%%%%%%%%%%%%%%%%%%%%%%%%%%%%%%%%%%%%%%%%%%%%%%%%%%%%%%%%%%%%%%%%%%%%%%% 
% Figure 2
%%%%%%%%%%%%%%%%%%%%%%%%%%%%%%%%%%%%%%%%%%%%%%%%%%%%%%%%%%%%%%%%%%%%%%%% 
\begin{figure}[ht] 
%\vskip.3truecm 
\centering 
\ins{112pt}{-35pt}{$v_0$} 
\ins{277pt}{-35pt}{$v_D$} 
\includegraphics[width=2.1in]{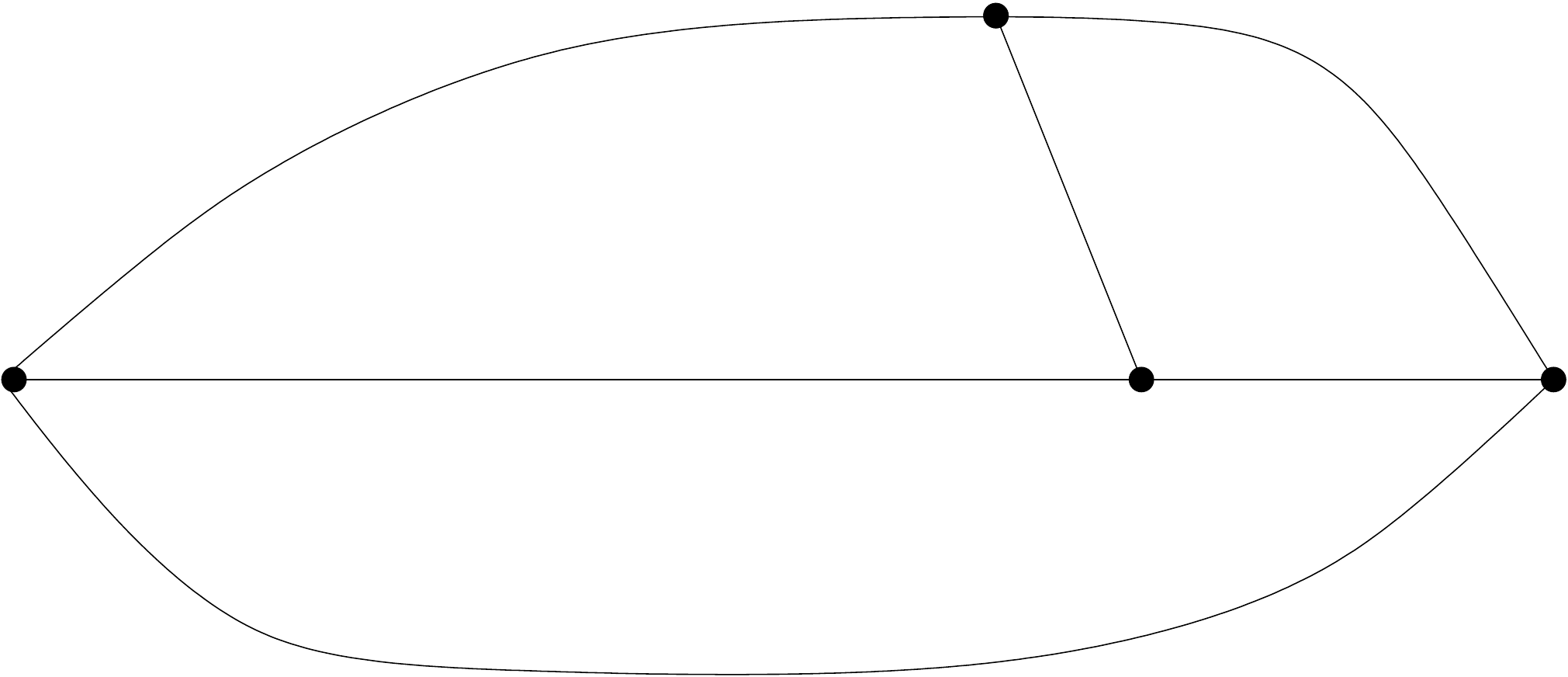} 
\vskip.2truecm 
\caption{The graph $G'$, created by removing pendants from $G$. } 
\label{fig2} 
\end{figure} 
%%%%%%%%%%%%%%%%%%%%%%%%%%%%%%%%%%%%%%%%%%%%%%%%%%%%%%%%%%%%%%%%%%%%%%%% 

\item[Step 6.] We  now construct a new graph $G''$ by altering $G'$ so that all paths connecting $v_0$ to $v_D$ have equal length $D$.   We start by
setting $\Gamma_1'' = \Gamma_1$.  Then, in order to shorten $\Gamma_2$ without changing $\Gamma_1''$, we reduce
$\Gamma_2 \backslash \Gamma_1$ to create a new path $\Gamma_2''$ with $\ell(\Gamma_2'')=D$. 
Proceeding by induction, we shorten each path $\Gamma_i$ to create $\Gamma_i''$ such that $\ell(\Gamma_i'')=D$, 
by only decreasing the length of $\Gamma_i\backslash \cup_{j=1}^{i-1} \Gamma_j$, until $\ell(\Gamma_i'')=D$. 
Note that some paths may have become subsets of others, leaving us $m$ distinct paths with $m \leq n$.
We set $G'' = \cup_{i=1}^m \Gamma_i''$.

\item[Step 7.]
For each path $\Gamma_i''$ in $G''$, define parametrizations 
$f_i:\Gamma_i'' \rightarrow$ $[0,D]$ with $f_i(v_0)=0$ and $f_k(v_D)=D$.
For each vertex $v_j$ in $\Gamma_1''$, we insure that $f_i^{-1}(f_1(v_j))$ is a vertex of $\Gamma_i''$ for $i > 1$
by inserting artificial vertices of degree two if needed. 
This process is repeated for $i=2,\dots,m$, so that all paths $\Gamma_i''$ have vertices corresponding to the same set of points in $[0,D]$.

 %%%%%%%%%%%%%%%%%%%%%%%%%%%%%%%%%%%%%%%%%%%%%%%%%%%%%%%%%%%%%%%%%%%%%%%% 
% Figure 3
%%%%%%%%%%%%%%%%%%%%%%%%%%%%%%%%%%%%%%%%%%%%%%%%%%%%%%%%%%%%%%%%%%%%%%%% 
\begin{figure}[ht] 
%\vskip.3truecm 
\centering 
\ins{112pt}{-35pt}{$v_0$} 
\ins{180pt}{-67pt}{$e_1$} 
\ins{250pt}{-58pt}{$e_4$} 
\ins{180pt}{-30pt}{$e_2$} 
\ins{250pt}{-30pt}{$e_5$} 
\ins{180pt}{4pt}{$e_3$} 
\ins{215pt}{7pt}{$v_2$} 
\ins{250pt}{2pt}{$e_6$} 
\ins{230pt}{-65pt}{$v_1$} 
\ins{230pt}{-40pt}{$v_3$} 
\ins{228pt}{-18pt}{$e_7$} 
\ins{277pt}{-35pt}{$v_D$} 
\includegraphics[width=2.1in]{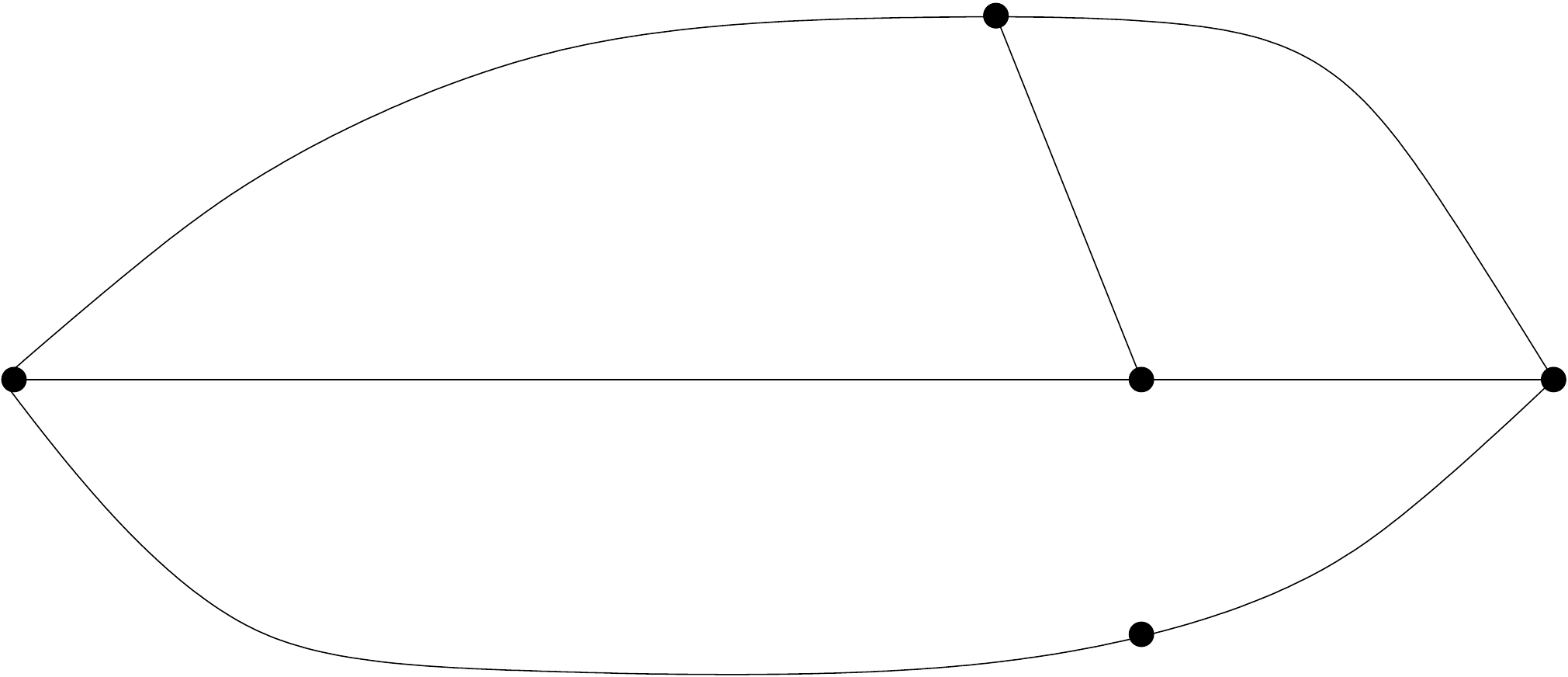} 
\vskip.2truecm 
\caption{The graph $G''$. Compared to Figure \ref{fig2}, the artificial vertex $v_1$ has been added, and the sets of
edges $\{e_1, e_2, e_3\}$ and $\{e_4, e_5, e_6\}$ have been altered so that their elements have the same lengths.} 
\label{fig3} 
\end{figure} 
%%%%%%%%%%%%%%%%%%%%%%%%%%%%%%%%%%%%%%%%%%%%%%%%%%%%%%%%%%%%%%%%%%%%%%%% 

\item[Step 8.] The final step is to identify all vertices with the same level under $f$. 
That is, for each $x_j \in [0,D]$ corresponding to a vertex, the points $f_i^{-1}(x_j)$ for $i = 1,\dots, m$
are identified in $G''$.  This leaves the pumpkin chain $G^*$.

 %%%%%%%%%%%%%%%%%%%%%%%%%%%%%%%%%%%%%%%%%%%%%%%%%%%%%%%%%%%%%%%%%%%%%%%% 
% Figure 4
%%%%%%%%%%%%%%%%%%%%%%%%%%%%%%%%%%%%%%%%%%%%%%%%%%%%%%%%%%%%%%%%%%%%%%%% 
\begin{figure}[ht] 
%\vskip.3truecm 
\centering 
\ins{112pt}{-31pt}{$v_0$} 
\ins{180pt}{-59pt}{$e_1$} 
\ins{252pt}{-47pt}{$e_4$} 
\ins{180pt}{-26pt}{$e_2$} 
\ins{252pt}{-26pt}{$e_5$} 
\ins{180pt}{7pt}{$e_3$} 
\ins{252pt}{-8pt}{$e_6$} 
\ins{230pt}{-40pt}{$v_3$} 
\ins{277pt}{-31pt}{$v_D$} 
\includegraphics[width=2.1in]{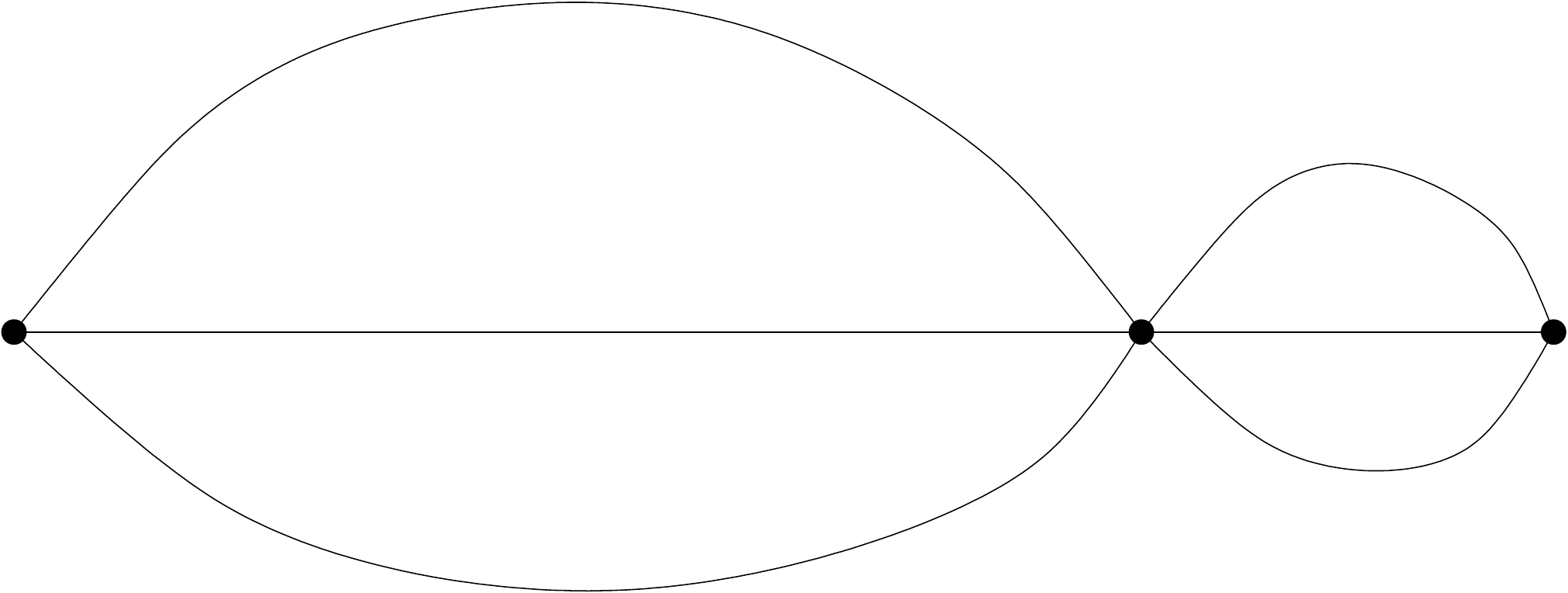} 
\vskip.2truecm 
\caption{The pumpkin chain $G^*$.  Compared to Figure \ref{fig3}, the vertices $v_1$, $v_2$, and $v_3$ have been identified,
and the edge $e_7$ has shrunk to a point.} 
\label{fig4} 
\end{figure} 
%%%%%%%%%%%%%%%%%%%%%%%%%%%%%%%%%%%%%%%%%%%%%%%%%%%%%%%%%%%%%%%%%%%%%%%% 

\end{itemize}

The construction of $G^*$ involves cutting pendants, shortening edge lengths, and identifying vertices. From Lemma \ref{known} we know that
these operations at most increase the first eigenvalue.  Therefore, we have $\lambda_1(G) \leq \lambda_1(G^*)$. 
\end{proof}

Note that the algorithm used in Lemma~\ref{pchain.lemma} does not claim uniqueness.  
In fact many different pumpkin chains might be achievable from the same original graph.

%%%%
\section{Spectral gap for pumpkin chains}\label{pchain.bnd.sec}

In this section we will establish the optimal upper bound on the spectral gap for pumpkin chains.  
Let $G$ be a pumpkin chain of total length $\ell$, with $m$ pumpkins.
We use the coordinate $x \in [0,\ell]$ to indicate the longitudinal position on the chain (taking the 
same values on each edge within the same pumpkin).

The calculation of the first eigenfunction for a pumpkin chain reduces to a one-dimensional 
problem.  Let $I_j$ be the open subinterval of $[0,\ell]$ corresponding to pumpkin $j$, 
and let $k_j$ denote its multiplicity.  We define the edge weight function
\begin{equation}\label{edge.wt}
\rho(x) := k_j\quad \text{for }x \in I_j.
\end{equation}
It was shown in \cite{KKMM:2016}--Lemma 5.6 that we can assume that the eigenfunction corresponding to $\lambda_1$ is a function of $x$ alone.
This implies that the Rayleigh quotient formula \eqref{lam1.rayleigh} reduces to
\begin{equation}\label{pchain.rayleigh}
\lambda_1(G) = \inf \left\{ \frac{\int_0^\ell \abs{f'}^2\rho\>dx}{\int_0^\ell \abs{f}^2\rho\>dx}:\>f \in H^1(0,\ell), \int_0^\ell f \rho \>dx = 0 \right\}.
\end{equation}
 
Since $\rho$ is piecewise constant, the minimizing function $\phi_1$ for \eqref{pchain.rayleigh} will 
take the form $b_j \cos (\sigma x + \alpha_j)$ in each segment, where $\sigma^2 = \lambda_1$
and $\alpha_j$ is a constant phase shift.  Furthermore, $\phi_1$ will satisfy Neumann conditions
at the endpoints $0$ and $\ell$, and the vertex conditions reduce to the statement that 
$\rho \phi_1'$ is continuous on $[0, \ell]$.  

To motivate the upper bound, let us first consider the case $m=2$.  
Let $x \in [0,\ell]$ the longitudinal coordinate, with the first pumpkin corresponding to $[0,\ell_1]$.
Let $k_1$, $k_2$ be the corresponding numbers of edges in each pumpkin.  
We seek an eigenfunction of the form
\[
\phi(x) = \begin{cases} b_1 \cos (\sigma x), &x \in [0,\ell_1], \\
b_2 \cos (\sigma (\ell-x)), & x\in [\ell_1,\ell], \end{cases}
\]
where $\lambda = \sigma^2$.
The coefficients $b_j$ are chosen to satisfy the continuity and vertex conditions
\begin{equation}\label{bj.eqs}
\begin{split}
b_1 \cos (\sigma \ell_1) &= b_2 \cos (\sigma(\ell-\ell_1)), \\
b_1 k_1 \sigma \sin (\sigma \ell_1) &= b_2 k_2 \sigma \sin (\sigma (\ell-\ell_1)).
\end{split}
\end{equation}
As a linear system for the coefficients $(b_1,b_2)$, \eqref{bj.eqs}
admits a nontrivial solution for if and only if 
\begin{equation}\label{m2.root}
k_1 \sin (\sigma \ell_1)  \cos (\sigma(\ell-\ell_1)) - k_2 \cos (\sigma \ell_1) \sin (\sigma (\ell-\ell_1)) = 0.
\end{equation}
The value of $\lambda_1$ is determined by the smallest positive root $\sigma_1$ of \eqref{m2.root}.

If $k_1$ and $k_2$ are roughly equal, then the first root of \eqref{m2.root} occurs when $\sigma \approx \pi/\ell$.
To maximize $\lambda_1$, we can assume that $k_1$ is much larger than $k_2$.  
In this case, the roots of \eqref{m2.root} will lie close to the roots of the first term, $\sin (\sigma \ell_1)  \cos (\sigma(\ell-\ell_1))$.  In particular
the first nonzero root $\sigma_1$ satisfies
\[
\sigma_1 \approx \min\left\{\frac{\pi}{\ell_1},  \frac{\pi}{2(\ell-\ell_1)}\right\}.
\]
The maximum on the right-hand side occurs when $\ell_1 = 2\ell/3$, suggesting that the maximum value of 
$\sigma_1$ is $3\pi/(2\ell)$ when $m=2$. 
 
\begin{theorem}\label{pchain.bnd.thm}
Let $G$ be a pumpkin chain of total length $\ell$, with $m$ pumpkins ($m+1$ vertices).
Then for $m \ge 2$,
\[
\lambda_1(G) \le \frac{(m+1)^2\pi^2}{4\ell^2}.
\]
\end{theorem}

\begin{proof}
Let us label the pumpkins $1, \dots, m$ so that the lengths $\ell_1,\dots,\ell_m$ are arranged in decreasing order.
For pumpkin $j$ let $k_j$ be the number of edges and $x_j$ the starting location.  

We first define a test function as in \cite{KKMM:2016}--Thm 6.1, by fitting half a wavelength of a cosine function into the longest pumpkin,
\[
\psi_1(x) := \begin{cases} b_1, & x \le x_1, \\
b_1 \cos [\pi(x-x_1)/\ell_1], & x_1\le x \le x_1 + \ell_1/2, \\
b_2 \cos [\pi(x-x_1)/\ell_1], & x_1+\ell_1/2 \le x \le x_1 + \ell_1, \\
-b_2, & x \ge x_1 + \ell_1,
\end{cases}
\]
with $b_1$ and $b_2$ chosen so that
\begin{equation}\label{ortho.psi1}
\int_0^\ell \psi_1 \rho\>dx = 0.
\end{equation}
This test function is illustrated in Figure~\ref{psi1test.fig}.

\begin{figure} 
\begin{center}  
\begin{overpic}[scale=.9]{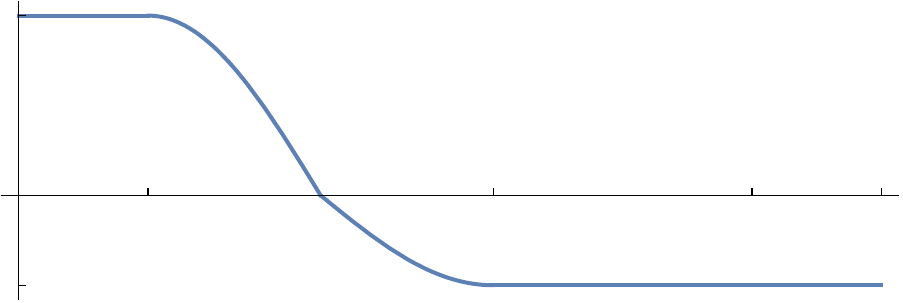} 
\put(15,9){$x_1$}
\put(53,9){$x_2$}
\put(97,8){$\ell$}
\end{overpic}
\end{center}
\caption{The test function $\psi_1$.}\label{psi1test.fig}
\end{figure}

The numerator of the Rayleigh quotient for $\psi_1$ is
\[
\int_0^\ell \abs{\psi_1'}^2 dx = \frac{\pi^2 k_1}{2\ell_1} (b_1^2 + b_2^2).
\]
For the denominator, we include only the contribution from pumpkin $j$ to obtain a lower bound,
\[
\int_0^\ell \abs{\psi_1}^2 dx \ge \frac{\ell_1k_1 }{2} ( b_1^2 +  b_2^2).
\]
By \eqref{pchain.rayleigh}, these estimates imply that
\begin{equation}\label{ell1.est}
\lambda_1 \le  \frac{\pi^2}{\ell_1^2}
\end{equation}

The next step is to construct an alternate test function by fitting a quarter-wavelength cosine into each of 
the two longest pumpkins, as illustrated in Figure~\ref{psi2test.fig}.
Without loss of generality, we can assume that $x_1 < x_2$ and define
\[
\psi_2(x) := \begin{cases} b_1, & x \le x_1, \\
b_1 \cos [\pi(x-x_1)/2\ell_2], & x_1\le x \le x_1 + \ell_2, \\
0, & x_1+\ell_2 \le x \le x_2, \\
-b_2 \sin [\pi(x - x_2)/2\ell_2], & x_2 \le x \le x_2+ \ell_2, \\
-b_2, & x \ge x_2 + \ell_2,
\end{cases}
\]
where once again $b_1$ and $b_2$ are chosen so that $\psi_2$ satisfies the orthogonality condition.

\begin{figure} 
\begin{center}  
\begin{overpic}[scale=.9]{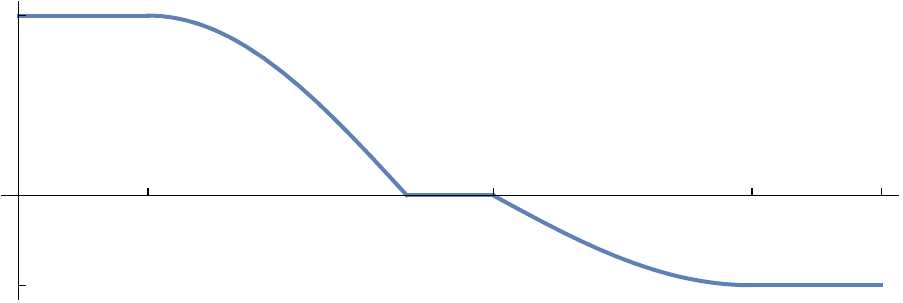} 
\put(15,9){$x_1$}
\put(53,9){$x_2$}
\put(97,8){$\ell$}
\end{overpic}
\end{center}
\caption{The test function $\psi_2$.} \label{psi2test.fig}
\end{figure}

For this second test function we compute
\[
\int_0^\ell \abs{\psi_2'}^2 dx = \frac{\pi^2}{8\ell_2} (k_1b_1^2 + k_2b_2^2),
\]
and estimate
\[
\int_0^\ell \abs{\psi_2}^2 dx \ge \frac{\ell_2}{2} (k_1b_1^2 + k_2b_2^2)
\]
By \eqref{lam1.rayleigh},
\begin{equation}\label{ell2.est}
\lambda_1 \le \frac{\pi^2}{4\ell_2^2}.
\end{equation}

To complete the argument, we note that the pumpkins $2,\dots, m$ have lengths at most $\ell_2$ and add up to a total length of $\ell-\ell_2$.
Thus the first two lengths satisfy an inequality,
\[
(m-1)\ell_2 \ge \ell-\ell_1.
\]
Plugging this into \eqref{ell2.est} gives
\begin{equation}\label{ell3.est}
\lambda_1 \le  \frac{\pi^2(m-1)^2}{4(\ell-\ell_1)^2}.
\end{equation}
In combination with \eqref{ell1.est}, we this have
\begin{equation}\label{opt.ell1.bnd}
\lambda_1 \le \pi^2 \left( \min\left\{\frac{1}{\ell_1}, \frac{m-1}{2(\ell-\ell_1)}\right\} \right)^2.
\end{equation}
The right-hand side attains a maximum when $\ell_1 = 2\ell/(m+1)$, yielding the claimed estimate.
\end{proof}

Theorem~\ref{main.thm} now follows immediately from Lemma~\ref{pchain.lemma} and Theorem~\ref{pchain.bnd.thm}.

\section{Sharpness of the estimate}\label{sharp.sec}

Let $G$ be a pumpkin chain with total length $\ell$ and $m$ components, as in \S\ref{pchain.bnd.sec}.  
Our goal in this section is to show that the bound in Theorem~\ref{pchain.bnd.sec} is sharp.  
From the estimate \eqref{opt.ell1.bnd}, it is clear that the maximum value of $\lambda_1$ 
should occur when $\ell_1 = 2\ell/(m+1)$ and $\ell_2 = \ell/(m+1)$.  Note that this implies that all
other segments also have length $\ell/(m+1)$.  

\begin{lemma}\label{sharp.lemma}
Suppose $G$ is a pumpkin chain with length $\ell$ and $m \ge 2$ components.  For $a := \ell/(m+1)$, 
assume that one segment has length $2a$ and all others have length $a$.  Given $\vep>0$, there is
a choice of edge multiplicities $k_1,\dots, k_m$ such that
\[
\lambda_1 > \frac{\pi^2}{4a^2} - \vep.
\]
\end{lemma}

\begin{proof}
For $\delta>0$ small, we will show how to choose edge multiplicities to produce an eigenvalue $\lambda = \sigma_1^2$ where 
\[
\sigma_1 = \frac{\pi}{2(a+\delta)}.
\]
Note that in an interval of length $a$, $\cos (\sigma_1 x)$ passes through just under a quarter period.  

Let $x_j$ denote the starting point of the $j$th pumpkin, for $j=1,\dots, m$.  Let $j_0$ denote the segment 
of length $2a$.  The strategy is to splice together functions of the form 
\begin{equation}\label{hj.segments}
h_j(x) := \begin{cases} \cos (\sigma_1(x - x_j) + \eta_j), &1 \le j \le j_0, \\
\sin (\sigma_1(x - x_j) + \eta_j), &j_0 < j \le m. \end{cases}
\end{equation}
where each phase shift $\eta_j$ is an integer multiple of $\sigma_1 \delta/2$.  We will specify the phase shifts,
and then use the vertex conditions to fix the multiplicities $k_j$.

The recipe for choosing $\eta_j$ is as follows.  
If $j \ne j_0$, then we set 
\[
\eta_j := \begin{cases} 0, &\text{if }j = 1, \\
\sigma_1 \delta/2, &\text{if }1<j<m, \\
\sigma_1 \delta, &\text{if } j=m.  \end{cases}
\]
For the segment of double length, $j = j_0$, the phase shifts are doubled,
\[
\eta_{j_0} := \begin{cases} 0, &\text{if }j_0 = 1, \\
\sigma_1 \delta, &\text{if }1<j_0<m, \\
2\sigma_1 \delta, &\text{if } j_0=m.  \end{cases}
\]

The full eigenfunction $\phi$, is defined by setting
\[
\phi(x) := b_j h_j(x), \quad\text{for }x \in [x_j, x_j + \ell_j],
\]
with $b_j$ defined by vertex and continuity conditions.  The matching conditions at vertex $x_j$ are
\begin{equation}\label{bj.cond}
\begin{split}
b_{j-1} h_{j-1}(x_j) &= b_j h_j(x_j), \\
b_{j-1}k_{j-1} h'_{j-1}(x_j) &= b_jk_j h'_j(x_j),
\end{split}
\end{equation}
for $j=2, \dots, m$.  Hence the edge multiplicities satisfy the condition
\begin{equation}\label{kj.cond}
k_{j-1} h'_{j-1}(x_j)h_j(x_j) = k_j h'_j(x_j)h_{j-1}(x_j), \quad 2 \le j \le m.
\end{equation}
If we choose $\delta$ so that $\sin(\sigma_1 \delta/2)$ and $\cos(\sigma_1 \delta/2)$ are both rational, then by
basic trig identities, all of the values of $h$ and $h'/\sigma$ appearing in \eqref{kj.cond} will be rational.  Hence
there exist integers $k_1,\dots, k_m$ satisfying the relation.
We can find arbitrarily small values of $\delta$ satisfying the rational condition by choosing a large integer $n$ and
setting
\[
\frac{\sigma_1\delta}2 = \frac{\pi \delta}{4(a+\delta)} = \arctan \left(\frac{2n}{n^2-1}\right).
\]

\begin{figure} 
\begin{center}  
\begin{overpic}[scale=.9]{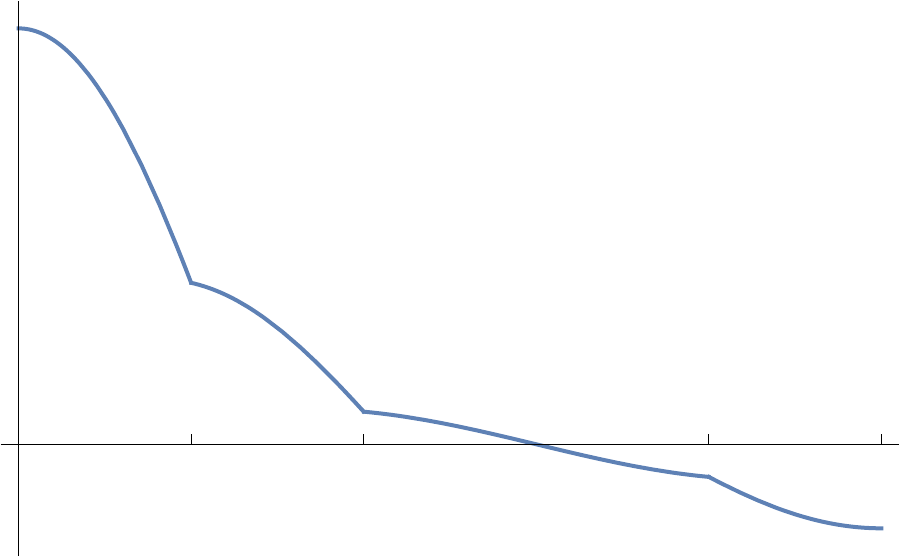} 
\put(20,10){$x_1$}
\put(39,10){$x_2$}
\put(77,10){$x_3$}
\end{overpic}
\end{center}
\caption{A sample constructed eigenfunction $\phi$.} \label{phiplot.fig}
\end{figure}

Assume that $\delta$ is chosen so that $\sigma_1 \delta < \pi/8$,  
After choosing the corresponding integers $k_1,\dots, k_m$ satisfying \eqref{kj.cond}.
we can then solve the coefficient equations \eqref{bj.cond} to construct an eigenfunction $\phi$ with eigenvalue $\lambda = \sigma_1^2$.  
One can easily check that under the assumption that $\sigma_1 \delta < \pi/8$, $\phi$ will be strictly decreasing, with a single zero that
occurs at the midpoint of the double-length segment, $x_{j_0}+a$.  An example is illustrated in Figure~\ref{phiplot.fig}, with $m=4$ and $j_0=3$.

To prove that $\phi$ corresponds to the eigenvalue $\lambda_1$, one approach would be to adapt the classical Sturm-Liouville theorem, 
which says that the $n$th eigenfunction has exactly $n$ zeros if $\rho \in C^1[0,\ell]$.  For convenience, we give a direct 
argument.  An eigenfunction $\psi$ with $0< \sigma < \sigma_1$ must start with a segment 
proportional to $\cos (\sigma x)$ at $x=0$.  Furthermore, since the wavelength associated to $\sigma$ is longer, the segments $h_j$ for $j \le j_0$
take the same form as \eqref{hj.segments}, with $\sigma_1$ replaced by $\sigma$ and with different choices of the phase shifts $\eta_j$.  if we rescale 
$\psi$ so that $\psi(0) = \phi(0)$, then since the matching conditions \eqref{bj.cond} are the same for both functions, we can deduce that $\psi> \phi$
through segment $j_0$.  In particular, $\psi$ has no zero for $0 \le x \le x_{j_0}$.  Applying the same reasoning in reverse, starting from $x=\ell$,
shows that $\psi$ has no zeros for $x_{j_0} \le x \le \ell$ either.  This is a contradiction, since $\psi$ is continuous and orthogonal to constant functions.
Therefore, the eigenfunction $\phi_1$ corresponds to the lowest nonconstant eigenfunction.  We have thus produced a pumpkin chain with
\[
\lambda_1 = \frac{\pi^2}{4(a+\delta)^2},
\]
where $\delta$ is arbitrarily small.
\end{proof}

In the $m=4$ case pictured in Figure~\ref{phiplot.fig}, it is easy to work out that the choices of $k_j$ indicated in the construction 
from the lemma are proportional to the values $\{1,N,N^2,2N\}$,  for some large number $N$.  This gives a relatively easy way to construct
optimal examples in this case.  For example, taking $k_1=1$, $k_2 = 10^{10}$, $k_3 = 10^{20}$, and $k_4 = 2\times 10^{10}$ gives the first eigenvalue
$\sigma_1 \doteq 2.49998\pi$, very close to the optimal value of $5\pi/2$.

%bibtex version:
%\bibliographystyle{amspl-db}
%\bibliography{ref.bib}

\begin{thebibliography}{10}

\bibitem{BG:2017}
R. Band and G. L\'{e}vy, Quantum graphs which optimize the spectral gap,
  \emph{Ann. Henri Poincar\'{e}} \textbf{18} (2017), 3269--3323.

\bibitem{BK:2012}
G. Berkolaiko and P. Kuchment, Dependence of the spectrum of a quantum graph on
  vertex conditions and edge lengths, \emph{Spectral geometry}, Proc. Sympos.
  Pure Math., vol.~84, Amer. Math. Soc., Providence, RI, 2012, pp.~117--137.

\bibitem{BKKM:2017}
G. Berkolaiko, J.~B. Kennedy, P. Kurasov, and D. Mugnolo, Edge connectivity and
  the spectral gap of combinatorial and quantum graphs, \emph{J. Phys. A}
  \textbf{50} (2017), 365201, 29.

\bibitem{BKKM:2018}
G. Berkolaiko, J.~B. Kennedy, P. Kurasov, and D. Mugnolo, Surgery principles
  for the spectral analysis of quantum graphs, preprint, {\tt
  arXiv:1807.08183}, 2018.

\bibitem{BK:2013}
G. Berkolaiko and P. Kuchment, \emph{Introduction to Quantum Graphs},
  Mathematical Surveys and Monographs, vol. 186, American Mathematical Society,
  Providence, RI, 2013.

\bibitem{Friedlander:2005}
L. Friedlander, Extremal properties of eigenvalues for a metric graph,
  \emph{Ann. Inst. Fourier (Grenoble)} \textbf{55} (2005), 199--211.

\bibitem{Kennedy:2018}
J.~B. Kennedy, A sharp eigenvalue bound for quantum graphs in terms of the
  diameter, preprint, {\tt arXiv:1807.08185}, 2018.

\bibitem{KKMM:2016}
J.~B. Kennedy, P. Kurasov, G. Malenov\'{a}, and D. Mugnolo, On the spectral gap
  of a quantum graph, \emph{Ann. Henri Poincar\'{e}} \textbf{17} (2016),
  2439--2473.

\bibitem{KMN:2013}
P. Kurasov, G. Malenov\'{a}, and S. Naboko, Spectral gap for quantum graphs and
  their edge connectivity, \emph{J. Phys. A} \textbf{46} (2013), 275309, 16.

\bibitem{KN:2014}
P. Kurasov and S. Naboko, Rayleigh estimates for differential operators on
  graphs, \emph{J. Spectr. Theory} \textbf{4} (2014), 211--219.

\bibitem{Nicaise:1987}
S. Nicaise, Spectre des r\'{e}seaux topologiques finis, \emph{Bull. Sci. Math.
  (2)} \textbf{111} (1987), 401--413.

\end{thebibliography}

%explicit bibliography for submission:

\end{document}